\documentclass[12pt,a4paper,oneside]{amsart}
\usepackage{amsfonts, amsmath, amssymb, amsthm}
\usepackage{mathtools}
\usepackage{anysize}
\marginsize{2cm}{2cm}{1.5cm}{1.5cm}

\usepackage[utf8]{inputenc}
\usepackage[english]{babel}
\usepackage{cmap}
\usepackage{enumerate}
\usepackage{xcolor}
\usepackage{hyperref}
\hypersetup{
	colorlinks   = true,
	urlcolor     = blue,
	linkcolor    = blue,
	citecolor    = red
}
\usepackage{cancel}

\theoremstyle{theorem}
\newtheorem{thm}{Theorem}[section]
\newtheorem{prp}[thm]{Proposition}
\newtheorem{lem}[thm]{Lemma}

\theoremstyle{definition}
\newtheorem{dfn}{Definition}[section]
\newtheorem{rem}[dfn]{Remark}

\newcommand{\Href}[2]{\hyperref[#2]{#1~\ref{#2}}}

\newcommand{\st}{:\;}

\def\N{{\mathbb N}}
\def\R{{\mathbb R}}

\newcommand{\ball}[1]{\mathbf{B}^{#1}}

\newcommand{\dist}[2]{\operatorname{dist}\!\left( #1, #2 \right)}

\newcommand{\iprod}[2]{\left\langle#1,#2\right\rangle}

\newcommand{\norm}[1]{\left\|#1\right\|}

\providecommand{\card}[1]{\lvert#1\rvert}

\newcommand{\conv}{\mathrm{conv}}

\providecommand{\parenth}[1]{\left(#1\right)}
\providecommand{\braces}[1]{\left\{#1\right\}}

\def\epsilon{\varepsilon}

\title{On Banach Spaces with the Helly Approximation Property}
\author{Grigory Ivanov}

\begin{document}

\begin{abstract}
Qualitatively, a no-dimensional Helly-type theorem says that if every small subfamily of convex sets has a common point in a bounded region, then suitable neighborhoods of all the sets in the whole family have a common point. Quantitative bounds, when available, depend on the ambient metric. We say that a Banach space has the Helly approximation property if the radii of these neighborhoods tend to zero as the size of the subfamilies tends to infinity.

In this paper, we show that the Helly approximation property holds if and only if the dual space has non-trivial Rademacher type. The argument combines Maurey's empirical method with a duality argument at a minimizer of the maximal distance function. We also prove a colorful version of this theorem, with control over the average of the radii.
\end{abstract}
\maketitle

\section{Introduction}

In \cite{adiprasito2020theorems}, Adiprasito, Bárány, Mustafa, and Terkaj introduced and proved the following no-dimensional version of the celebrated Helly theorem:
\medskip

\noindent
\emph{If a finite family of convex sets in a Euclidean space has the property that every $k$-element subfamily has a common point inside the unit ball, then there exists a point at distance at most \(1/\sqrt{k}\) from all the sets.}

\medskip

Their proof uses a simple but clever geometric argument based on the Pythagorean theorem: it estimates the distance between the origin and the intersection of two convex sets \(K\) and \(L\) in terms of the distance from the origin to \(K\) and the distance from \(K\) to \(L\).
Later, the author extended this result to uniformly convex Banach spaces \cite{Ivanov2025nodimHelly}, replacing the Pythagorean theorem by estimates involving the modulus of convexity.

Let us formalize the problem for an arbitrary Banach space, following \cite{ivanov2026nodimensionalresults}.

Recall that the \(R\)-neighborhood of a set \(A\) is the set
\[
\braces{x \in X \st \dist{x}{A} \le R}.
\]

\begin{dfn}\label{dfn:helly_sequence}
Let \(X\) be a Banach space.
A sequence \(\braces{r_k(X)}_{k \in \N}\) is called a \emph{Helly approximation sequence} of \(X\) if it has the following property:

\smallskip
\noindent
For every finite family \(\mathcal{F}\) of convex subsets of \(X\), if every subfamily of \(\mathcal{F}\) of size \(k\) has a common point inside the unit ball, then the \(r_k(X)\)-neighborhoods of all sets in \(\mathcal{F}\) have a common point.
\end{dfn}

\begin{dfn}\label{dfn:helly_approx_property}
We say that a Banach space \((X,\norm{\cdot})\) has the \emph{Helly approximation property} if there exists a Helly approximation sequence \(\braces{r_k(X)}_{k \in \N}\) such that
\[
r_k(X)\to 0 \qquad\text{as } k\to\infty.
\]
\end{dfn}

Thus, \cite[Theorem~1.2]{adiprasito2020theorems} states that \(\braces{1/\sqrt{k}}\) is a Helly approximation sequence for Euclidean spaces, and in \cite{Ivanov2025nodimHelly} it was shown that uniformly convex Banach spaces have the Helly approximation property, with the corresponding sequences estimated in terms of the modulus of convexity.

Two related problems have remained open (see \cite[Problem~4.1]{ivanov2026nodimensionalresults}). The first is to characterize Banach spaces with the Helly approximation property; the second is to provide a probabilistic proof of the no-dimensional Helly theorem. In this paper we resolve both problems.

Recall that a Banach space \(X\) is said to be of 
\emph{ (Rademacher) type \(p\)} for some \(p\ge 1\) if there exists a constant \(T_p(X)\) such that, for every finite set \(S\) of vectors in \(X\), the average norm of the \(2^{|S|}\) vectors of the form \(\sum_{s\in S}\pm s\) does not exceed
\[
T_p(X)\parenth{\sum_{s\in S}\norm{s}^p }^{\!1/p};
\]
see also~\cite[Definition~1.e.12]{lindenstrauss2013classical}.
We denote by \(T_p(X)\) the minimal constant in this inequality.

By the triangle inequality, every Banach space is of type~1. It is well known that the type of a Banach space cannot exceed~2, and that the maximal type of \(\ell_p\), \(p \in (1,+\infty)\), is \(\min\{2,p\}\).

We will say that a space is of \emph{non-trivial type} if its maximal type is strictly greater than one, and of \emph{trivial type} otherwise.

We denote the unit ball of a Banach space \(X\) by \(\ball{}_X\), 
$\iprod{p}{x}$ denotes the value of a functional $p$ on $x$, $[n]$ denotes the set $\braces{1, \dots, n}$.

Our main results are as follows.

\begin{thm}[No-dimensional Helly]\label{thm:no_dim_Helly_dual_type}
Let \(X\) be a Banach space such that \(X^*\) has Rademacher type \(p\in(1,2]\) with constant \(T_p(X^*)\). Set
\[
r_k(X)= 6 T_p(X^*)\,k^{-1+\frac{1}{p}}.
\]
Let \(\mathcal{F}\) be a finite family of convex sets in \(X\). Assume that for every choice of \(k\) sets
\(
K_1, \dots, K_k \in \mathcal{F}
\)
the intersection
\(
\ball{}_X\cap \bigcap_{i=1}^k K_i
\)
is non-empty.
Then there exists a point \(x\in X\) such that
\[
\dist{x}{K}\le r_k(X)
\qquad\text{for all } K\in\mathcal{F}.
\]

In particular, a Banach space whose dual is of non-trivial type has the Helly approximation property.
\end{thm}

We obtain this theorem as a straightforward corollary of the following colorful version.

\begin{thm}[Colorful no-dimensional Helly]\label{thm:colorful_no_dim_Helly_dual_type}
Let \(X\) be a Banach space such that \(X^*\) has Rademacher type \(p\in(1,2]\) with constant \(T_p(X^*)\). Set
\[
r_k(X)=6T_p(X^*)\,k^{-1+\frac{1}{p}}.
\]
Let \(\mathcal{F}_1,\dots,\mathcal{F}_k\) be finite families of convex sets in \(X\). Assume that for every ``rainbow'' choice
\(
K_i \in \mathcal{F}_i,\ i \in[k],
\)
the intersection
\(
\ball{}_X \cap \bigcap_{i=1}^k K_i
\)
is non-empty.
Then there exist radii \(r_1, \dots, r_k\) with
\[
\frac{r_1 + \dots + r_k}{k} \leq r_k(X)
\]
and a point \(x\in X\) such that
\[
\dist{x}{K}\le r_i
\qquad\text{for all } K\in\mathcal{F}_{i},\ i\in[k].
\]

In particular, there is a color \(i_0\in[k]\) such that
\[
\dist{x}{K}\le r_k(X)
\qquad\text{for all } K\in\mathcal{F}_{i_0}.
\]
\end{thm}

\begin{rem}
Our version of the colorful no-dimensional Helly theorem is stronger than the earlier versions in \cite{adiprasito2020theorems} and \cite{Ivanov2025nodimHelly}, since it provides an estimate for the average of the radii.
\end{rem}

We note that the dual space appears naturally, since we reduce the problem to a dual one, namely a no-dimensional Carathéodory lemma. We introduce the necessary notation and explain this duality in the next section. Then, we  prove the no-dimensional Helly theorems. 
Finally, we obtain the desired characterization by providing a counterexample in \Href{Section}{sec:counterexample_Helly_trivial_type}:

\begin{thm}\label{thm:Helly_approx_prop_characterization}
The following properties are equivalent for a Banach space \(X\):
\begin{enumerate}
\item \(X\) has the Helly approximation property;
\item \(X^*\) is of non-trivial type;
\item \(X\) is of non-trivial type.
\end{enumerate}
\end{thm}

We prove in this paper only the equivalence of the first two properties.
The equivalence of \(X^*\) being of non-trivial type and \(X\) being of non-trivial type is not new.
It was pointed out to the author by Alexandros Eskenazis.
It follows from the celebrated result of Pisier \cite[Theorem~2.1]{pisier1982holomorphic}
on the equivalence between non-trivial type and \(K\)-convexity, together with the standard fact
that a Banach space is \(K\)-convex if and only if its dual is \(K\)-convex.
We also refer to \cite{fackler2011holomorphic} for a proof of this duality fact and for a clear exposition of the topic.

In particular, it follows that the Helly approximation property holds precisely in the same class of Banach spaces in which the Maurey sampling argument is available, namely, the spaces of non-trivial type. 

Finally, let us discuss the finite dimensional case. In \cite{ivanov2026nodimensionalresults}, the author showed that several classical $\ell_1$-approximation problem admit ``local-to-global'' estimates via no-dimensional Helly-type results. 
The argument from \cite{ivanov2026nodimensionalresults} yields
\[
r_k \! \parenth{\ell_1^n}\le C \sqrt{\frac{\ln n}{k}},
\]
by approximating \(\ell_1^n\) by \(\ell_q^n\) for a suitable exponent \(q=q(n)\) and applying the deterministic approach from \cite{Ivanov2025nodimHelly}; on the dual side, this corresponds to replacing \(\ell_\infty^n\) by \(\ell_p^n\), where \(p\) and \(q\) are conjugate exponents. The same estimate also follows from \Href{Theorem}{thm:no_dim_Helly_dual_type}, since \((\ell_1^n)^*=\ell_\infty^n\) and
\[
T_2\!\parenth{\ell_\infty^n}\le C\sqrt{\ln n}.
\]
We conjecture that the bound
\(
r_k\!\parenth{\ell_1^n}\le C \sqrt{\frac{\ln n}{k}}
\)
is optimal.

On the other hand, our examples show that
\[
r_k\!\parenth{\ell_\infty^n}\ge \frac{k}{2k-1}\ge \frac12
\qquad\text{whenever } 2k\le n.
\]
Thus, in \(\ell_\infty^n\) no decay in \(k\) is possible.

\section{Idea behind the proof, type, and the no-dimensional Carathéodory lemma}

\subsection{Reduction to a problem in the dual space}

In this subsection we explain the main idea behind the proof of \Href{Theorem}{thm:no_dim_Helly_dual_type}. We deliberately suppress all technicalities and focus only on the underlying mechanism.

The first simplification is a standard reduction to the finite-dimensional case. The only subtle point is the estimate for the type constant of the dual of a finite-dimensional subspace of the original space. This is a simple exercise from the definition of type, and we provide the proof in the next subsection for completeness.

Now let \(\mathcal{F}\) be our family of convex sets. We want to estimate the smallest radius \(r\) such that some translate of \(r\ball{}\) intersects every set in \(\mathcal{F}\). Equivalently, we consider the convex \(1\)-Lipschitz function
\[
f(x)=\max_{K\in\mathcal{F}} \dist{x}{K},
\]
and choose a point \(p\) at which \(f\) attains its minimum. The value
\[
r=f(p)
\]
is precisely the minimal radius we seek.

In the Euclidean case, the point \(p\) necessarily belongs to the convex hull of the ``active contact points'' on the sphere of radius \(r\) centered at $p$ and on the boundaries of some sets in the family. This folklore observation can be used to derive the classical Helly theorem \cite{helly1923mengen} in \(\R^d\) from the classical Carathéodory lemma \cite{caratheodory1911variabilitatsbereich}. We translate this necessary condition to the Banach-space setting. Using standard facts from convex analysis, we will show that instead of vectors in the original space one should consider functionals in the dual space, namely subgradients of the distance functions.

Since \(p\) is a minimizer, subdifferential calculus gives
\[
0 \in \partial f(p).
\]
Write
\[
\mathcal{A}:=\braces{K\in\mathcal{F} \st \dist{p}{K}=r}
\]
for the family of \emph{active} sets. Since \(f\) is the maximum of the distance functions \(x\mapsto \dist{x}{K}\), the subdifferential of \(f\) at \(p\) is the convex hull of the subdifferentials of the active distance functions. Thus there exist active sets \(K_1,\dots,K_m\in\mathcal{A}\) and functionals
\[
u_1,\dots,u_m\in \ball{}_{X^*}
\]
such that
\[
0\in \conv\braces{u_1,\dots,u_m},
\]
where $\conv$ denotes the convex hull.

After translating the picture, we may assume that \(p=0\). Then, for every active set \(K_i\), the corresponding functional \(u_i\) satisfies
\[
\iprod{u_i}{x}\le -r
\qquad\text{for all } x\in K_i.
\]
In other words, each \(K_i\) is contained in a supporting half-space at distance \(r\) from the origin.

At this point the geometric problem has been converted into a dual one. Suppose that we can find indices \(i_1,\dots,i_k\) such that
\[
\norm{\frac{1}{k}\sum_{t=1}^k u_{i_t}}
\]
is small. By the Helly assumption, the intersection
\[
\ball{}\cap \bigcap_{t=1}^k K_{i_t}
\]
is non-empty; let \(q\) be a point in it. It follows from the triangle inequality, that the norm of the shifted point \(q\) is bounded by an absolute constant, and
\[
\iprod{u_{i_t}}{q}\le -r
\qquad\text{for every } t\in[k].
\]
Averaging, we obtain
\[
r
\le
-\frac{1}{k}\sum_{t=1}^k \iprod{u_{i_t}}{q}
\le
\norm{\frac{1}{k}\sum_{t=1}^k u_{i_t}}\,\norm{q}.
\]
Thus any upper bound on the norm of such an average immediately yields an upper bound on the Helly radius \(r\).

This reduces the proof of \Href{Theorem}{thm:no_dim_Helly_dual_type} to the following dual approximation problem:

\medskip
\noindent
\emph{Given vectors \(u_1,\dots,u_m\in \ball{}_{X^*}\) with \(0\in\conv\braces{u_1,\dots,u_m}\), find \(k\) of them whose average has small norm.}

\medskip

This is precisely a no-dimensional Carathéodory-type statement in the dual space.

The colorful version is obtained in exactly the same spirit. Instead of one function \(f\), one considers a maximal distance function for each color and minimizes their average. The minimizer again produces active subgradients, now grouped by colors, whose total sum is zero. Thus the colorful Helly theorem is reduced to a colorful Carathéodory-type lemma in the dual space.
\subsection{No-dimensional Carathéodory lemma}

A no-dimensional Carathéodory lemma in a Banach space of non-trivial type is known as the celebrated Maurey's lemma \cite{pisier1980remarques}:

\begin{lem}[Maurey's lemma]\label{lem:Maurey_for_Helly}
Let \(Y\) be a Banach space of Rademacher type \(p\in(1,2]\) with constant \(T_p(Y)\). Assume that \(u_1,\dots,u_m\in\ball{}_Y\) and \(0 \in \conv\{u_1,\dots,u_m\}\). Then for every \(k\in\N\) there exist indices \(i_1,\dots,i_k\in[m]\) such that
\[
\norm{
\frac{1}{k}\sum_{t=1}^k u_{i_t}}
\le
2T_p(Y)\,k^{-1+\frac{1}{p}}.
\]
\end{lem}

It admits a one-line proof. For the sake of completeness, we provide the proof of the following colorful variant, which will be used in the proof of \Href{Theorem}{thm:colorful_no_dim_Helly_dual_type} (see also \cite[Lemma~2.1]{ivanov2021no}):

\begin{lem}[Colorful Maurey's lemma]\label{lem:colorful_dual_Maurey_Helly}
Let \(Y\) be a Banach space of type \(p\in(1,2]\) with constant \(T_p(Y)\). For each \(i\in[k]\), let
\(
u_{i,1},\dots,u_{i,m_i}\in \ball{}_Y
\)
be such that
\(
a_i\in \conv\braces{u_{i,1},\dots,u_{i,m_i}},
\)
and assume that
\[
\sum_{i=1}^k a_i=0.
\]
Then there exist indices \(t_i \in [m_i]\), \(i\in[k]\), such that
\[
\norm{\frac1k\sum_{i=1}^k u_{i,t_i}}
\le 2T_p(Y)\,k^{-1+\frac1p}.
\]
\end{lem}

\begin{proof}
For each \(i\in[k]\), choose non-negative coefficients \(\lambda_{i,1},\dots,\lambda_{i,m_i}\ge 0\) such that
\[
\sum_{j=1}^{m_i}\lambda_{i,j}=1
\qquad\text{and}\qquad
a_i=\sum_{j=1}^{m_i}\lambda_{i,j}u_{i,j}.
\]
Let \(U_1,\dots,U_k\) be independent random vectors such that
\[
\mathbb{P}\parenth{U_i=u_{i,j}}=\lambda_{i,j}
\qquad\text{for } i\in[k], \ j\in[m_i].
\]
Then
\[
\mathbb{E}U_i=a_i
\qquad\text{and}\qquad
\norm{U_i} \le 1
\]
almost surely. Since \(\sum_{i=1}^k a_i=0\), we have
\begin{equation}
\label{eq:col_Maurey_average_expec}
\mathbb{E}\parenth{\frac1k\sum_{i=1}^k U_i}=0.
\end{equation}

Let \(U_1',\dots,U_k'\) be an independent copy of \(U_1,\dots,U_k\), and let \(\varepsilon_1,\dots,\varepsilon_k\) be independent Rademacher random variables, independent of everything else. Then
\begin{align*}
\mathbb{E}\norm{\frac1k\sum_{i=1}^k U_i}
&\stackrel{\eqref{eq:col_Maurey_average_expec}}{=}
\mathbb{E}\norm{\frac1k\sum_{i=1}^k \parenth{U_i-\mathbb{E}U_i'}} \stackrel{(\text{Jensen})}{\le}
\mathbb{E}\norm{\frac1k\sum_{i=1}^k \parenth{U_i-U_i'}} \stackrel{\text{(symm.)}}{=}
\mathbb{E}\norm{\frac1k\sum_{i=1}^k \varepsilon_i \parenth{U_i-U_i'}} \\
&\le
\frac1k\,\mathbb{E}\norm{\sum_{i=1}^k \varepsilon_i U_i}
+
\frac1k\,\mathbb{E}\norm{\sum_{i=1}^k \varepsilon_i U_i'} =
\frac{2}{k}\,\mathbb{E}\norm{\sum_{i=1}^k \varepsilon_i U_i},
\end{align*}
where in step \((\text{Jensen})\) we used Jensen's inequality, and step \(\text{(symm.)}\) follows from symmetry.
Using Jensen's inequality and the type \(p\) property of \(Y\), we obtain
\begin{align*}
\mathbb{E}\norm{\sum_{i=1}^k \varepsilon_i U_i}
&\le
\parenth{\mathbb{E}\norm{\sum_{i=1}^k \varepsilon_i U_i}^p}^{1/p} \le
\parenth{\mathbb{E}\braces{T_p(Y)^p \sum_{i=1}^k \norm{U_i}^p}}^{1/p} \le
T_p(Y)\,k^{1/p}.
\end{align*}
Therefore,
\[
\mathbb{E}\norm{\frac1k\sum_{i=1}^k U_i}
\le
2T_p(Y)\,k^{-1+\frac1p}.
\]
Hence there exists a realization \(U_i=u_{i,t_i}\), \(i\in[k]\), such that
\[
\norm{\frac1k\sum_{i=1}^k u_{i,t_i}}
\le
2T_p(Y)\,k^{-1+\frac1p}.
\]
\end{proof}

\begin{rem}
Maurey's lemma guarantees that the corresponding averages converge to the origin. It is now understood that such an approximation property (which one may call the Carathéodory approximation property) holds if and only if the space is of non-trivial type \cite{artstein2025b}. Together with our yet-to-be-proven \Href{Theorem}{thm:Helly_approx_prop_characterization}, this shows a certain duality between no-dimensional Helly- and Carathéodory-type results. As was explained in detail in \cite{Ivanov2025nodimHelly}, deterministic proofs of no-dimensional Helly and Carathéodory results \cite{ivanov2021approximate} use inequalities dual to each other and work in classes of spaces related by duality as well.
\end{rem}

\subsection{Properties of the Rademacher type}

In our proofs we will reduce everything to finite-dimensional subspaces. For this reason, we need two auxiliary facts. The first one shows that, when we pass to a finite-dimensional subspace, the type constant of the dual does not increase. The second one is the Maurey--Pisier result from \cite{maurey1976series} in the form needed for our counterexample.

\begin{lem}\label{lem:type_of_dual_subspace}
Let \(X\) be a Banach space, and let \(E \subset X\) be a finite-dimensional subspace. Assume that \(X^*\) has Rademacher type \(p \in [1,2]\). Then \(E^*\) has Rademacher type \(p\), and
\[
T_p\parenth{E^*}\le T_p\parenth{X^*}.
\]
\end{lem}

\begin{proof}
Consider the annihilator
\[
E^\perp:=\braces{\phi\in X^* \st \phi|_E=0}.
\]
It is standard that \(E^*\) is isometrically isomorphic to the quotient \(X^*/E^\perp\). Indeed, define
\[
Q:X^*\to E^*,
\qquad
Q(\phi):=\phi|_E.
\]
Then \(Q\) is onto, \(\ker Q=E^\perp\), and, by the Hahn--Banach theorem \cite[Chapter~3]{rudin1991functional}, the induced map
\[
\widetilde{Q}:X^*/E^\perp \to E^*
\]
is an isometric isomorphism.

Thus, it is enough to show that the Rademacher type \(p\) constant does not increase when passing to a quotient. Let \(Y\) be a Banach space of type \(p\), let \(M\subset Y\) be a closed subspace, and let
\[
\pi:Y\to Y/M
\]
be the quotient map. By the definition of the quotient norm,
\[
\norm{\pi}\le 1.
\]
Take arbitrary vectors \(z_1,\dots,z_N\in Y/M\), and fix \(\varepsilon>0\). For every \(i\in[N]\), choose \(y_i\in Y\) such that
\[
\pi y_i=z_i
\qquad\text{and}\qquad
\norm{y_i}\le \norm{z_i}+\varepsilon.
\]
Then, for independent Rademacher random variables \(\varepsilon_1,\dots,\varepsilon_N\), we have
\begin{align*}
\mathbb{E}\norm{\sum_{i=1}^N \varepsilon_i z_i}^p
&=
\mathbb{E}\norm{\pi\parenth{\sum_{i=1}^N \varepsilon_i y_i}}^p \le
\norm{\pi}^p\,\mathbb{E}\norm{\sum_{i=1}^N \varepsilon_i y_i}^p \le
\mathbb{E}\norm{\sum_{i=1}^N \varepsilon_i y_i}^p \\
&\le
T_p(Y)^p \sum_{i=1}^N \norm{y_i}^p \le
T_p(Y)^p \sum_{i=1}^N \parenth{\norm{z_i}+\varepsilon}^p.
\end{align*}
Letting \(\varepsilon\to 0\), we obtain
\[
\mathbb{E}\norm{\sum_{i=1}^N \varepsilon_i z_i}^p
\le
T_p(Y)^p \sum_{i=1}^N \norm{z_i}^p.
\]
Hence
\[
T_p\parenth{Y/M}\le T_p(Y).
\]

Applying this to \(Y=X^*\) and \(M=E^\perp\), and using the isometric identification
\[
E^*\cong X^*/E^\perp,
\]
we conclude that \(E^*\) has type \(p\) and
\[
T_p\parenth{E^*}\le T_p\parenth{X^*}.
\]
\end{proof}

By the Maurey--Pisier result \cite[Theorem~2.1]{maurey1976series}, a Banach space has non-trivial Rademacher type if and only if it does not contain \(\ell_1^n\) uniformly. Equivalently, if it has only trivial type, then the spaces \(\ell_1^n\) are uniformly finitely representable in it. Suppressing the formal definitions, the following statement is a consequence of \cite[Theorem~2.1]{maurey1976series}:

\begin{prp}\label{prp:trivial_type_l1_representation}
Let \(Y\) be a Banach space of trivial type. Then for every \(\eta>0\) and every \(m\in\N\) there exists an \(m\)-dimensional subspace \(F\subset Y\) and an isomorphism
\[
T:\ell_1^m \to F
\]
such that
\[
\norm{T}\le 1
\qquad\text{and}\qquad
\norm{T^{-1}}\le 1+\eta.
\]
\end{prp}

\section{Proofs of no-dimensional Helly-type results}

\begin{proof}[Proof of \Href{Theorem}{thm:colorful_no_dim_Helly_dual_type}]

The proof consists of several simple steps.

\medskip
\noindent
\textbf{Step 1. Reduction to the finite-dimensional case.}

First, we simplify the picture by replacing each convex set by the convex hull of finitely many witness points.
We replace each set \(K\) from one of the families by the convex hull of points
\(x_{K, K_1, \dots, K_{k-1}}\) from the intersection \(\ball{}_X \cap K \cap \bigcap\limits_{i=1}^{k-1} K_i\), where,  to form the \(k\)-tuple of sets \(\{K, K_1, \dots, K_{k-1}\}\), we pick one set from each family.
Thus, each \(K\) becomes a compact convex polytope inside \(\ball{}_X\).
Let \(E\) be the linear hull of all newly constructed sets.
It is a finite-dimensional subspace of \(X\).

By \Href{Lemma}{lem:type_of_dual_subspace}, it suffices to find a point \(c\) inside \(\ball{}_E\) such that the families deviate from it by at most
\[
6T_p(E^*)\,k^{-1+\frac{1}{p}}
\]
on average in order to prove \Href{Theorem}{thm:colorful_no_dim_Helly_dual_type}.

From now on, we work in the finite-dimensional space \(E\), and our families consist of convex polytopes inside the unit ball \(\ball{}_E\).

\medskip
\noindent
\textbf{Step 2. Necessary conditions for the best ball.}

We recall a few facts from convex analysis \cite{rockafellar1970convex}.
We use the standard notion of \emph{subdifferential}:

If \(g:E\to\R\) is convex, then the subdifferential \(\partial g(x)\) at \(x\) is defined by
\[
\partial g(x):=
\braces{
\psi\in E^* \st
g(y)\ge g(x)+\iprod{\psi}{y-x}
\text{ for all } y\in E };
\]
elements of the subdifferential are called \emph{subgradients}.

We will use the following standard facts:
\begin{enumerate}
\item if \(x\) minimizes \(g\), then \(0\in\partial g(x)\);
\item if \(g_1,\dots,g_m\) are convex and continuous, then
\[
\partial(g_1+\cdots+g_m)(x)
=
\partial g_1(x)+\cdots+\partial g_m(x);
\]
\item if \(g=\max\braces{g_1,\dots,g_m}\), then
\[
\partial g(x)
=
\conv\Bigl(
\bigcup_{j\in I(x)} \partial g_j(x)
\Bigr),
\]
where
\[
I(x):=\braces{j\in[m]\st g_j(x)=g(x)}.
\]
\end{enumerate}

\medskip
We continue with the proof of the theorem.
Define
\[
f_i(x):=\max_{K\in\mathcal{F}_i}\dist{x}{K},
\]
and set
\[
F(x):=\frac{1}{k} \sum_{i=1}^k f_i(x).
\]

Since \(E\) is finite-dimensional, we conclude that \(F\) attains its minimum at some point \(x_0\).
By the triangle inequality and the bounds \(f_i(0)\le 1\) for all \(i\in[k]\), we conclude that \(x_0 \in 2\ball{}_E\).

Define \(r_i = f_i(x_0)\) and the ``active'' subfamilies
\[
\mathcal{A}_i
:=
\braces{K\in\mathcal{F}_i \st \dist{x_0}{K}= r_i},
\qquad i\in[k].
\]

Because \(x_0\) minimizes \(F\), we have
\[
0\in \partial F(x_0).
\]
Using the sum rule for subdifferentials,
\[
\partial F(x_0)
=
\frac1k\bigl(\partial f_1(x_0)+\cdots+\partial f_k(x_0)\bigr).
\]
Hence there exist functionals
\(
a_i\in\partial f_i(x_0),
\ i\in[k], 
\)
 \ such that \
\(
\sum\limits_{i=1}^k a_i=0.
\)

Since each \(f_i\) is the maximum of distance functions (which are convex), for every \(i\in[k]\) there exist active sets
\[
K_{i,1},\dots,K_{i,m_i}\in\mathcal{A}_i
\]
and subgradients
\[
\psi_{i,1} \in \partial d_{K_{i,1}}(x_0),
\qquad
\dots,
\qquad
\psi_{i,m_i} \in \partial d_{K_{i,m_i}}(x_0),
\]
where \(d_{K_{i,j}}\) denotes the distance function \(\dist{\cdot}{K_{i,j}}\), such that
\[
a_i\in \conv\braces{\psi_{i,1},\dots,\psi_{i,m_i}}.
\]
Additionally, since the distance function to a convex set is \(1\)-Lipschitz, all the resulting \(\psi_{i,j}\) belong to the unit ball \(\ball{}_{E^*}\).

Second, for every \(x\in K_{i,j}\) we have \(d_{K_{i,j}}(x)=0\). Since \(\psi_{i,j}\in\partial d_{K_{i,j}}(x_0)\), the subgradient inequality yields
\[
0=d_{K_{i,j}}(x)
\ge
d_{K_{i,j}}(x_0)+\iprod{\psi_{i,j}}{x-x_0}
=
r_i+\iprod{\psi_{i,j}}{x-x_0}.
\]
Hence,
\begin{equation}
\label{eq:subgrad_ineq_colorful_Helly}
\iprod{\psi_{i,j}}{x-x_0}\le -r_i
\qquad\text{for all } x\in K_{i,j}.
\end{equation}

\medskip
\noindent
\textbf{Step 3. Use of the no-dimensional Carathéodory lemma.}

Now apply \Href{Lemma}{lem:colorful_dual_Maurey_Helly} in the space \(E^*\) to the families
\[
\braces{\psi_{i,1},\dots,\psi_{i,m_i}},
\qquad i\in[k],
\]
with the points \(a_i\). Since \(\sum_{i=1}^k a_i=0\), there exist indices
\(
t_i\in[m_i],
\ i\in[k],
\)
such that
\[
\norm{
\frac1k\sum_{i=1}^k \psi_{i,t_i}
}
\le
2T_p(E^*)\,k^{-1+\frac1p}
\stackrel{\Href{Lemma}{lem:type_of_dual_subspace}}{\le}
2T_p(X^*)\,k^{-1+\frac1p}
=
\frac13\,r_k(X).
\]

For every \(i\in[k]\), the set \(K_{i,t_i}\) belongs to \(\mathcal{F}_i\). By the hypothesis, there exists a ``rainbow'' witness point
\[
q\in \ball{}_E\cap \bigcap_{i=1}^k K_{i,t_i}.
\]
Since \(\norm{x_0}\le 2\), we have
\[
\norm{q-x_0}\le \norm{q}+\norm{x_0}\le 3.
\]

Averaging inequalities \eqref{eq:subgrad_ineq_colorful_Helly}, we obtain
\[
\frac1k\sum_{i=1}^k r_i
\le
-\frac1k\sum_{i=1}^k \iprod{\psi_{i,t_i}}{q-x_0}
=
\iprod{
\frac1k\sum_{i=1}^k \psi_{i,t_i}
}{x_0-q}.
\]
Hence
\[
\frac1k\sum_{i=1}^k r_i
\le
\Bigl\|
\frac1k\sum_{i=1}^k \psi_{i,t_i}
\Bigr\|
\,\norm{x_0-q}
\le
\frac13\,r_k(X)\cdot 3
=
r_k(X).
\]
The desired inequality follows.
Finally, one of the radii, say \(r_{i_0}\), is at most the average.
This proves the theorem.
\end{proof}

\begin{rem}
Using explicit bounds in the no-dimensional Carathéodory lemma in the Euclidean case, it follows that the average of the radii is at most \(\frac{1}{\sqrt{k}}\) in the Euclidean case, without any additional multiplicative constants.
\end{rem}

\subsection{Proof of the non-colorful theorem}

\begin{proof}[Proof of \Href{Theorem}{thm:no_dim_Helly_dual_type}]
Apply \Href{Theorem}{thm:colorful_no_dim_Helly_dual_type} to the coinciding families
\[
\mathcal{F}_1=\cdots=\mathcal{F}_k=\mathcal{F}.
\]
Then every ``rainbow'' choice is simply a choice of \(k\) sets from \(\mathcal{F}\), so the hypothesis of the colorful theorem becomes exactly the hypothesis of \Href{Theorem}{thm:no_dim_Helly_dual_type}. 
\end{proof}

\section{Counterexample when the dual is of trivial type}
\label{sec:counterexample_Helly_trivial_type}

In this section we show that the assumption that \(X^*\) has non-trivial type is  necessary for the Helly approximation property. We first construct an explicit example in \(\ell_\infty^{2k}\), and then transfer it to an arbitrary Banach space whose dual has only trivial type.

\subsection{A counterexample in \(\ell_\infty\)}

\begin{lem}\label{lem:linfty_counterexample}
For every \(k \in \N\) there exist compact convex sets
\(
K_1,\dots,K_{2k} \subset \ell_\infty^{2k}
\)
such that the intersection
\(
\ball{}_{\ell_\infty^{2k}}\cap\bigcap_{i \in J} K_i 
\) is non-empty
{for every } 
\(
J \subset [2k] \text{ with } \card{J}=k,
\)
but
\[
\inf_{y \in \ell_\infty^{2k}}\max_{i \in [2k]} \dist{y}{K_i} \geq \frac{k}{2k-1}.
\]
In particular, there is no no-dimensional Helly approximation sequence in \(\ell_\infty\) that tends to \(0\).
\end{lem}

\begin{proof}
Let \(e_1,\dots,e_{2k}\) be the standard unit vector basis of \(\ell_1^{2k}=(\ell_\infty^{2k})^*\). For \(i\in[2k]\), define
\[
w_i:=e_i-\frac{1}{2k}\sum_{j=1}^{2k}e_j.
\]
Then
\(
\sum_{i=1}^{2k} w_i=0
\qquad\text{and}\qquad
\norm{w_i}_1
=
\parenth{1-\frac{1}{2k}}+(2k-1)\frac{1}{2k}
=
\frac{2k-1}{k}.
\)

Set
\[
a_k:=\frac{k}{2k-1}
\qquad\text{and}\qquad
u_i:=a_k\,w_i.
\]
Then \(u_i\in(\ell_\infty^{2k})^*\), \(\norm{u_i}_1 = 1\), and
\(
\sum_{i=1}^{2k}u_i=0.
\)

For every \(J\subset[2k]\) with \(\card{J}=k\), define \(x_J\in\ell_\infty^{2k}\) by
\[
x_J(t)=
\begin{cases}
-1, & t\in J,\\
1, & t\notin J.
\end{cases}
\]
Clearly,
\(
\norm{x_J}_\infty=1
\)
and 
\(
\sum_{t=1}^{2k}x_J(t)=0.
\)
 Hence, if \(i\in J\), then
\[
\iprod{u_i}{x_J}
=
a_k\parenth{x_J(i)-\frac1{2k}\sum_{t=1}^{2k}x_J(t)}
=
-a_k.
\]

Now define
\[
K_i:=\conv\braces{x_J \st J\subset[2k],\ \card{J}=k,\ i\in J},
\qquad i\in[2k].
\]
Each \(K_i\) is a compact convex subset of \(\ball{}_{\ell_\infty^{2k}}\). Moreover, if \(J\subset[2k]\) and \(\card{J}=k\), then for every \(i\in J\) the point \(x_J\) is one of the vertices used in the definition of \(K_i\). Thus, the witness point $x_J$ belongs to 
\( 
\ball{}_{\ell_\infty^{2k}}\cap\bigcap_{i\in J}K_i,
\)
which proves the \(k\)-wise intersection property.

It remains to prove the lower bound on the approximation radius. By construction,
\[
\iprod{u_i}{z}\le -a_k
\qquad\text{for all } z\in K_i.
\]

Assume now that \(y\in \ell_\infty^{2k}\) and \(r>0\) satisfy
\[
\dist{y}{K_i}\le r
\qquad\text{for all } i\in[2k].
\]
For each \(i\), choose \(z_i\in K_i\) such that
\[
\norm{y-z_i}_\infty\le r.
\]
Since \(\norm{u_i}_1 = 1\) and \(\iprod{u_i}{z_i}\le -a_k\), we obtain
\[
\iprod{u_i}{y}
=
\iprod{u_i}{z_i}+\iprod{u_i}{y-z_i}
\le
-a_k+\norm{y-z_i}_\infty
\le
-a_k+r.
\]
Averaging over \(i\in[2k]\) and using \(\sum_{i=1}^{2k}u_i=0\), we get
\[
0
=
\iprod{\frac1{2k}\sum_{i=1}^{2k}u_i}{y}
\le
-a_k+r.
\]
Hence
\[
r\ge a_k=\frac{k}{2k-1}.
\]
Since \(y\) was arbitrary,
\[
\inf_{y \in \ell_\infty^{2k}}\max_{i \in [2k]} \dist{y}{K_i}
\ge \frac{k}{2k-1}.
\]
\end{proof}

\subsection{Spaces whose dual has only trivial type}

We now transfer the above construction to an arbitrary Banach space whose dual has only trivial type.

\begin{lem}\label{lem:trivial_type_counterexample}
Assume that \(X^*\) has only trivial type. Then for every \(\varepsilon>0\) and every \(k\in\N\) there exist compact convex sets
\(
K_1,\dots,K_{2k}\subset X
\)
such that the intersection
\(
\ball{}_{X}\cap\bigcap_{i \in J} K_i 
\) is non-empty
{for every } 
\(
J \subset [2k] \text{ with } \card{J}=k,
\)
but
\[
\inf_{y\in X}\max_{i\in[2k]}\dist{y}{K_i}
\ge \frac{k}{2k-1}-\varepsilon.
\]
In particular, no no-dimensional Helly approximation sequence in \(X\) can tend to \(0\).
\end{lem}

\begin{proof}
Set
\[
a_k:=\frac{k}{2k-1}.
\]
Fix \(\varepsilon>0\). Choose \(\eta,\delta>0\) so that
\[
\frac{a_k}{1+\eta}-\delta>a_k-\varepsilon.
\]
By \Href{Proposition}{prp:trivial_type_l1_representation}, there exists a \(2k\)-dimensional subspace \(E\subset X^*\) and an isomorphism
\(
T:\ell_1^{2k}\to E
\)
such that
\(
\norm{T}\le 1
\)
\ and \
\(
\norm{T^{-1}}\le 1+\eta.
\)

Let \(e_1,\dots,e_{2k}\) be the standard unit vector basis of \(\ell_1^{2k}\). As  in \Href{Lemma}{lem:linfty_counterexample}, define
\[
w_i:=e_i-\frac{1}{2k}\sum_{j=1}^{2k}e_j
\qquad\text{and}\qquad
u_i:=a_k\,w_i,
\qquad i\in[2k].
\]
Set
\(
\psi_i:=T u_i \in E\subset X^*.
\)\
Since \(\norm{T}\le1\), we have
\(
\norm{\psi_i}\le1
\)
\ for all \
\(
i\in[2k],
\)
\ and \
\(
\sum_{i=1}^{2k}\psi_i=0.
\)

For every \(J\subset[2k]\) with \(\card{J}=k\), define \(x_J^0\in\ell_\infty^{2k}=(\ell_1^{2k})^*\) by
\[
x_J^0(t)=
\begin{cases}
-1, & t\in J,\\
1, & t\notin J.
\end{cases}
\]
Then \(\norm{x_J^0}_\infty=1\), and, as before, if \(i\in J\), then
\(
\iprod{u_i}{x_J^0}=-a_k.
\)

Consider the adjoint map
\[
(T^{-1})^*:\ell_\infty^{2k}\to E^*.
\]
Since \(\norm{T^{-1}}\le1+\eta\), \
\(
\norm{(T^{-1})^*}\le1+\eta.
\)
Define
\[
y_J^0:=\frac{1}{1+\eta}(T^{-1})^*x_J^0 \in E^*.
\]
Then
\(
\norm{y_J^0}\le1,
\)
and for every \(i\in J\),
\[
\iprod{y_J^0}{\psi_i}
=
\frac{1}{1+\eta}\iprod{(T^{-1})^*x_J^0}{Tu_i}
=
\frac{1}{1+\eta}\iprod{x_J^0}{u_i}
=
-\frac{a_k}{1+\eta}.
\]

Let
\[
Q:X\to E^*,
\qquad
Q(x)(\phi):=\iprod{\phi}{x}
\quad (\phi\in E),
\]
be the canonical evaluation map. By Goldstine's theorem \cite{rudin1991functional}, \(Q(\ball{}_X)\) is weak\(^*\)-dense in \(\ball{}_{E^*}\). Since \(E\) is finite-dimensional, weak\(^*\)-density coincides with norm density. Therefore, for every \(J\subset[2k]\) with \(\card{J}=k\), there exists \(x_J\in\ball{}_X\) such that
\[
\norm{Qx_J-y_J^0}_{E^*}<\delta.
\]
In fact, one can say more: Banach spaces are locally reflexive. We also note that the points \(x_J\) are chosen in the ambient space \(X\) and need not lie in a subspace of dimension \(2k\). At this stage, however, we only need to control their evaluations on the finite-dimensional space \(E\). 

Hence, for every \(i\in J\),
\[
\iprod{\psi_i}{x_J}
=
{Q(x_J)}(\psi_i)
\le
\iprod{y_J^0}{\psi_i} + \delta \norm{\psi_i}
\leq 
-\frac{a_k}{1+\eta}+\delta.
\]
Set
\[
b:=\frac{a_k}{1+\eta}-\delta.
\]
Then \(b>a_k-\varepsilon\), and for every \(i\in J\),
\[
\iprod{\psi_i}{x_J}\le -b.
\]

Now define
\[
K_i:=\conv\!\braces{x_J \st J\subset[2k],\ \card{J}=k,\ i\in J},
\qquad i\in[2k].
\]
Each \(K_i\) is a compact convex subset of \(\ball{}_X\).

If \(J\subset[2k]\) and \(\card{J}=k\), then \(x_J\in K_i\) for every \(i\in J\). Thus,
\(
x_J\in \ball{}_X\cap\bigcap_{i\in J}K_i,
\)
so the \(k\)-wise intersection property holds.

By convexity,
\[
\iprod{\psi_i}{z}\le -b
\qquad\text{for all } z\in K_i.
\]

Now assume that \(y\in X\) and \(r>0\) satisfy
\[
\dist{y}{K_i}\le r
\qquad\text{for all } i\in[2k].
\]
For each \(i\), choose \(z_i\in K_i\) such that
\(
\norm{y-z_i}\le r.
\)
Since \(\norm{\psi_i}\le1\) and 
\(\iprod{\psi_i}{z_i}\le -b\), we obtain
\[
\iprod{\psi_i}{y}
=
\iprod{\psi_i}{z_i} + \iprod{\psi_i}{y-z_i} 
\le
-b + \norm{y-z_i}
\le
-b + r.
\]
Averaging over \(i\in[2k]\) and using \(\sum_{i=1}^{2k}\psi_i=0\), we get
\[
0
=
\iprod{\frac1{2k}\sum_{i=1}^{2k}\psi_i}{y}
\le
-b+r.
\]
Hence
\[
r\ge b>a_k-\varepsilon.
\]
Therefore,
\[
\inf_{y\in X}\max_{i\in[2k]}\dist{y}{K_i}
\ge b>a_k-\varepsilon.
\]
This proves \Href{Lemma}{lem:trivial_type_counterexample}.
\end{proof}

\subsection{Characterization}

\begin{proof}[Proof of \Href{Theorem}{thm:Helly_approx_prop_characterization}]
As mentioned in the Introduction, the equivalence of \(X^*\) being of non-trivial type and \(X\) being of non-trivial type follows from the celebrated result of Pisier \cite[Theorem~2.1]{pisier1982holomorphic}. Therefore, it suffices to prove that \(X\) has the Helly approximation property if and only if \(X^*\) is of non-trivial type.

The implication
\[
X^* \text{ is of non-trivial type } \Longrightarrow X \text{ has the Helly approximation property}
\]
is exactly \Href{Theorem}{thm:no_dim_Helly_dual_type}.

For the converse, \Href{Lemma}{lem:trivial_type_counterexample} shows that if \(X^*\) is of trivial type, then
\[
r_k(X)\ge \frac{k}{2k-1}
\]
for all \(k\). It follows that \(r_k(X)\) does not tend to \(0\). Hence, \(X\) does not have the Helly approximation property.
\end{proof}

\section*{Acknowledgements}
I thank Alexander Polyanskii for stimulating discussions that led to this probabilistic viewpoint. I am also very grateful to Alexandros Eskenazis, who pointed out several factual mistakes in the first version of the manuscript and explained to me several deep results about Rademacher type.

\bibliographystyle{alpha}
\bibliography{../work_current/uvolit}

\begin{thebibliography}{ABMT20}

\bibitem[ABMT20]{adiprasito2020theorems}
Karim Adiprasito, Imre B{\'a}r{\'a}ny, Nabil~H. Mustafa, and Tam{\'a}s Terpai.
\newblock Theorems of {C}arath{\'e}odory, {H}elly, and {T}verberg without
  dimension.
\newblock {\em Discrete \& Computational Geometry}, 64(2):233--258, 2020.

\bibitem[AK25]{artstein2025b}
Zvi Artstein and Vladimir Kadets.
\newblock {$B$}-convexity, {C}onvexification of {M}inkowski {A}verages in a
  {B}anach {S}pace, and {S}{L}{L}{N} for {R}andom {S}ets.
\newblock {\em Journal of Convex Analysis}, 32(1):61--70, 2025.

\bibitem[Car11]{caratheodory1911variabilitatsbereich}
Constantin Carath{\'e}odory.
\newblock {\"U}ber den {V}ariabilit{\"a}tsbereich der {F}ourierschen
  {K}onstanten von positiven harmonischen {F}unktionen.
\newblock {\em Rendiconti Del Circolo Matematico di Palermo (1884-1940)},
  32(1):193--217, 1911.

\bibitem[Fac11]{fackler2011holomorphic}
Stephan Fackler.
\newblock Holomorphic {S}emigroups and the {G}eometry of {B}anach {S}paces.
\newblock Diploma thesis, Universit\"at Ulm, April 2011.

\bibitem[Hel23]{helly1923mengen}
E.~Helly.
\newblock {\"U}ber mengen konvexer {K}{\"o}rper mit gemeinschaftlichen
  {P}unkte.
\newblock {\em Jahresbericht der Deutschen Mathematiker-Vereinigung},
  32:175--176, 1923.

\bibitem[Iva21a]{ivanov2021approximate}
Grigory Ivanov.
\newblock Approximate {C}arath{\'e}odory’s theorem in uniformly smooth
  {B}anach spaces.
\newblock {\em Discrete \& Computational Geometry}, 66(1):273--280, 2021.

\bibitem[Iva21b]{ivanov2021no}
Grigory Ivanov.
\newblock No-dimension {T}verberg's theorem and its corollaries in {B}anach
  spaces of type $p$.
\newblock {\em Bulletin of the London Mathematical Society}, 53(2):631--641,
  2021.

\bibitem[Iva25]{Ivanov2025nodimHelly}
Grigory Ivanov.
\newblock No-dimensional {H}elly’s theorem in uniformly convex {B}anach
  spaces.
\newblock {\em Studia Scientiarum Mathematicarum Hungarica}, April 2025.

\bibitem[Iva26]{ivanov2026nodimensionalresults}
Grigory Ivanov.
\newblock No-dimensional results of combinatorial convexity. {D}imension
  strikes back.
\newblock \emph{arXiv preprint arXiv:2602.20035}, 2026.

\bibitem[LT13]{lindenstrauss2013classical}
Joram Lindenstrauss and Lior Tzafriri.
\newblock {\em Classical Banach spaces II: function spaces}, volume~97.
\newblock Springer Science \& Business Media, 2013.

\bibitem[MP76]{maurey1976series}
Bernard Maurey and Gilles Pisier.
\newblock S{\'e}ries de variables al{\'e}atoires vectorielles ind{\'e}pendantes
  et propri{\'e}t{\'e}s g{\'e}om{\'e}triques des espaces de banach.
\newblock {\em Studia Mathematica}, 58(1):45--90, 1976.

\bibitem[Pis80]{pisier1980remarques}
Gilles Pisier.
\newblock Remarques sur un r{\'e}sultat non publi{\'e} de {B}. {M}aurey.
\newblock {\em S{\'e}minaire Analyse fonctionnelle (dit" Maurey-Schwartz")},
  pages 1--12, 1980.

\bibitem[Pis82]{pisier1982holomorphic}
Gilles Pisier.
\newblock Holomorphic semi-groups and the geometry of {B}anach spaces.
\newblock {\em Annals of Mathematics}, 115(2):375--392, 1982.

\bibitem[Roc70]{rockafellar1970convex}
Ralph~Tyrrell Rockafellar.
\newblock {\em Convex analysis}.
\newblock Number~28. Princeton university press, 1970.

\bibitem[Rud91]{rudin1991functional}
Walter Rudin.
\newblock {\em Functional Analysis}.
\newblock International Series in Pure and Applied Mathematics. McGraw-Hill,
  New York, 1991.

\end{thebibliography}

\end{document}